\documentclass[11pt, reqno]{amsart}

\usepackage{amssymb,latexsym,amsmath,amsfonts}
\usepackage{mathrsfs}
\usepackage{graphicx}
\usepackage[usenames]{color}
\usepackage{hyperref}

\definecolor{DPurple}{rgb}{0.46,0.2,0.69}

\numberwithin{equation}{section}

\allowdisplaybreaks

\theoremstyle{definition}
\newtheorem{definition}{Definition}[section]

\theoremstyle{remark}
\newtheorem{remark}[definition]{Remark}

 \theoremstyle{plain}
\newtheorem{theorem}[definition]{Theorem}
\newtheorem{result}[definition]{Result}

\newtheorem{lemma}[definition]{Lemma}

\newtheorem{example}[definition]{Example}
\newtheorem{corollary}[definition]{Corollary}



\begin{document}

\title[Meromorphically normal families]{Meromorphically 
normal families and \\ a meromorphic Montel--Carath{\'e}odory theorem}

\author{Gopal Datt}
\address{Department of Mathematics, Indian Institute of Science, Bangalore 560012, India}
\email{gopaldatt@iisc.ac.in}



\begin{abstract}
In this paper, we present various sufficient conditions for a family of meromorphic mappings on a
domain $D\subset \mathbb{C}^m$ into $\mathbb{P}^n$ to be meromorphically normal. Meromorphic
normality is a notion of sequential compactness in the meromorphic category introduced by
Fujimoto. We give a general condition for meromorphic normality that is influenced by Fujimoto's work.
The approach to proving this result allows us to establish meromorphic analogues of several recent
results on normal families of $\mathbb{P}^n$-valued holomorphic mappings. We also establish a
meromorphic version of the Montel--Carath{\'e}odory theorem.
\end{abstract}
\maketitle

\vspace{-0.6cm}
\section{Introduction and Main Results}\label{S:intro}

The work in this paper is influenced, chiefly, by two different types of results on
families of holomorphic or meromorphic mappings into
$n$-dimensional complex projective space, $\mathbb{P}^n$, $n\geq 2$.
The first influence is the work of Fujimoto, who introduced the
notion of {\it meromorphic convergence} (or {\it $m$-convergence}) for a 
sequence of meromorphic mappings from a domain $D\subset\mathbb{C}^m$ into
$\mathbb{P}^n$.
A family $\mathcal{F}$ of meromorphic mappings from 
$D\subset\mathbb{C}^m$ into $\mathbb{P}^n $ is said to be {\it meromorphically normal} if, roughly
speaking, any sequence in $\mathcal{F}$ has a subsequence $\{f_j\}$ with the property
that any point
$a\in D$ admits a neighborhood $U\ni a$ such that\,---\,fixing homogeneous coordinates
on $\mathbb{P}^n$\,---\,for each $f_j$ there is a reduced representation
$[f_{j,\,0}: f_{j,\,1}:\cdots: f_{j,\,n}]$ for $\left.f_j\right|_U$ and such that   
each $\{f_{j,\,l}\}$, $0\leq l\leq n$, converges compactly on $U$ 
to a holomorphic function $f_l$ and $f_l\not\equiv 0$ for at least one $l\in\{0, 1,\dots, n\}$.
Fujimoto called the latter notion of convergence {\it $m$-convergence}.
We shall discuss this notion, and what it means, more rigorously in 
Section~\ref{S:notions}. At this juncture, we merely remark that  $m$-convergence
is a more well-behaved mode of convergence than the notion of ``quasi-regular convergence''
introduced by Rutishauser~\cite{Rutishauser50}. Indeed, 
Fujimoto presented the following sufficient condition for meromorphic normality (among
other objectives) as an improvement of Rutishauser's result: 
 
\begin{result}[{\cite[Theorem~4.3]{Fujimoto74}}]\label{R: Fujithm1}
  Let $D$ be a domain in $\mathbb{C}^m$ and let $\mathcal{F}$ be a family of meromorphic mappings
  from $D$ into $\mathbb{P}^n$. Let $H_1,\dots, H_{2n+1}$ be $2n+1$
  hyperplanes in general position in $\mathbb{P}^n$ such that for each 
  $f\in\mathcal{F}$ and each $k\in\{1,\dots, 2n+1\}$, $ f(D)\not\subset H_k$. 
  Suppose that for each closed ball $\overline{B}\subset D$,
  the volumes of $f^{-1}(H_k)\cap \overline{B}$, viewing $f^{-1}(H_k)$ as divisors, 
  are uniformly bounded for
  $k=1,\dots, 2n+1$, and $f\in \mathcal{F}$.  
  Then the family $\mathcal{F}$ is meromoprhically normal.
\end{result}

In its general appearance\,---\,and in the use of the Kobayashi hyperbolicity of 
$\mathbb{P}^n\setminus\big(\cup_{k=1}^{2n+1}~H_k\big)$ ($H_1,\dots, H_{2n+1}$ are $2n+1$
hyperplanes in $\mathbb{P}^n$ as above) in its proof\,---\,the above result suggests the following natural
question: {\it Is there a version of the Montel--Carath{\'e}odory theorem
for families of meromorphic mappings from $D\subset \mathbb{C}^m$ into
$\mathbb{P}^n$?} To elaborate: the classical Montel--Carath{\'e}odory theorem states that  
a family of holomorphic $\mathbb{P}^1$-valued mappings on a planar domain is normal if this family omits
three fixed, distinct points in $\mathbb{P}^1$. This generalizes to higher dimensions: a family of holomorphic
$\mathbb{P}^n$-valued mappings on a domain $D\subset \mathbb{C}^m$ is normal if this family omits
$2n+1$ hyperplanes in $\mathbb{P}^n$ located in general position. This result is due to  
Dufresnoy~\cite{Dufresnoy} (also see Kiernan--Kobayashi \cite[Section~4]{Kiernan Kobayashi73}).
It is natural to ask whether, under the hypothesis of the latter result, a family of meromorphic
mappings into $\mathbb{P}^n$ is meromorphically normal. We answer this question in the
affirmative. This (see Corollary~\ref{C: Mon-Cara}) follows from a rather general criterion for
meromorphic normality. 
\smallskip

To motivate the criterion just alluded to, we must introduce the other result that
influences our work. To state this result, we need to introduce an essential quantity which\,---\,given
a collection of hyperplanes $H_1,\dots, H_q$ in $\mathbb{P}^n$, $q\geq n+1$,
in general position\,---\,quantifies in a
canonical way {\it to what extent} this collection is in general position. To this end, we   
fix a system of homogeneous coordinates $w= [w_0: w_1:\cdots: w_n]$ on $\mathbb{P}^n$,
whence any hyperplane $H$ in $\mathbb{P}^n$ can be given by
\begin{equation}\label{Eq: hyperplane}
  H := 
  \left\{[w_0: w_1:\cdots: w_n]\in\mathbb{P}^n\,\big| 
  \sum\nolimits_{l=0}^{n}a_{l}w_{l}=0\right\},
\end{equation}
where $(a_0, a_1,\dots, a_n)=:\alpha\in\mathbb{C}^{n+1}$ is a non-zero vector.
In particular, we can take $\alpha\in\mathbb{C}^{n+1}$ such that $\|\alpha\|=1$.
Let $H_1,\dots, H_{n+1}$ be hyperplanes in $\mathbb{P}^n$, and let 
$\alpha_k := \big(a_0^{(k)},\dots, a_n^{(k)}\big)$ be non-zero unit vectors in 
$\mathbb{C}^{n+1}$ such that for each $k\in\{1,\dots, n+1\}$, $H_k$ is given by
\begin{equation*}
  H_k :=
  \left\{[w_0: w_1:\cdots: w_n]\in\mathbb{P}^n\,\big| 
  \sum\nolimits_{l=0}^{n}a_l^{(k)}w_l=0\right\}.
\end{equation*}
We define
\begin{equation*}
  \mathcal{D}(H_1,\dots, H_{n+1}) :=
  \left|\det    
  \begin{bmatrix}
    \alpha_1 \\ 
    \vdots \\ 
    \alpha_{n+1}
  \end{bmatrix}
  \right|.
\end{equation*}
It is absolutely elementary to see\,---\,since $\alpha_1,\dots, \alpha_{n+1}$ are unit vectors\,---\,that
$\mathcal{D}(H_1,\dots, H_{n+1})$ depends only on $\{H_k : 1\leq k\leq n+1\}$ and is independent of the choice of
$\alpha_k$. Now let $H_1,\dots, H_q$ be hyperplanes in $\mathbb{P}^n$, where $q\geq n+1$. Set
\begin{equation*}
  D(H_1,\dots, H_q) := 
  \prod_{1\leq k_1<\dots<k_{n+1}\leq q}\mathcal{D}(H_{k_1},\dots, H_{k_{n+1}}).
\end{equation*}
Recall that the hyperplanes $H_1,\dots, H_q$ are {\it in general position} if $D(H_1,\dots, H_q)>0$.
\smallskip

We are now in a position to state the second result that influences the present work.
 
\begin{result}[{\cite[Theorem~2.10]{Yang Fang Pang 2014}}]\label{R: Yang Fang Pang 2014} 
  Let $D$ be a planar domain and $q\geq 3n+1$ be an integer. Let $\mathcal{F}$ and $\mathcal{G}$ 
  be two distinct families of holomorphic mappings from $D$ into $\mathbb{P}^n$.
  Suppose that the following conditions are satisfied:
  \begin{enumerate}
    \item[$(a)$] For each $f\in\mathcal{F}$, there exist a $g\in\mathcal{G}$ and
    hyperplanes $H_{1,\,f},\dots, H_{q,\,f}$ (which may depend on $f$)
    such that $f$ and $g$ share $H_{k,\,f}$ on $D$, i.e., $f^{-1}(H_{k,\,f})=g^{-1}(H_{k,\,f})$
    for $k= 1,\dots, q$;
  
    \item[$(b)$] $\inf\{D(H_{1,\,f},\dots, H_{q,\,f}): f\in\mathcal{F}\}>0$; and
  
    \item[$(c)$] The family $\mathcal{G}$ is normal.
  \end{enumerate}
  Then the family $\mathcal{F}$ is normal.
\end{result}

Two natural questions arise immediately in connection with the above result:
\begin{itemize}
  \item Given the close connection between the normality of a family of holomorphic mappings
  from a domain $D$ into some subset of $\mathbb{P}^n$ and the Kobayashi hyperbolicity
  of the latter subset, could the number of hyperplanes featured
  in Result~\ref{R: Yang Fang Pang 2014} be reduced and yet yield the same conclusion?
  \item Does a version of Result~\ref{R: Yang Fang Pang 2014} hold true for
  the domain $D\subset \mathbb{C}^m$, $m\geq 2$?
\end{itemize}
It turns out that the first question is slightly {\bf naive} since the hypothesis of
Result~\ref{R: Yang Fang Pang 2014} is absolutely lacking in information on the
extent to which $f(D)$ avoids $H_{k,\,f}$, $f\in \mathcal{F}$, $1\leq k\leq q$.
Moreover, Yang {\it et al.} show that the number of hyperplanes featured
in Result~\ref{R: Yang Fang Pang 2014} cannot, in general, be taken to be less than
$3n+1$: see \cite[Example~1]{Yang Fang Pang 2014}. As for the second
question: Yang, Liu and Pang in \cite{Yang Liu Pang 2016} present several improvements
to Result~\ref{R: Yang Fang Pang 2014}, one of which implies a version of the
latter result for mappings in $m$ variables $m\in \mathbb{N}$.
\smallskip

The last two observations led us to look for a variant of Result~\ref{R: Yang Fang Pang 2014}
wherein one considers a family $\mathcal{F}$ of {\it meromorphic} maps from
$D\subset \mathbb{C}^m$, $m\in \mathbb{N}$, into $\mathbb{P}^n$ and establishes
the meromorphic normality of $\mathcal{F}$. In this enterprise, we were influenced by
Result~\ref{R: Fujithm1}. {\bf Loosely} speaking, if we augment the hypothesis of
Result~\ref{R: Yang Fang Pang 2014} by requiring that  
$f(D)\not\subset H_{k,\,f}$, $f\in \mathcal{F}$\,---\,as in
Result~\ref{R: Fujithm1}\,---\,then one only needs to
consider $2n+1$ hyperplanes in $\mathbb{P}^n$ to play the role analogous to the above result.
These are the motivations for our first theorem: i.e., Theorem~\ref{T: mainthm2n+1} below.
Of course, given that $f\in \mathcal{F}$ may possess a non-empty indeterminacy locus, it is not
immediately clear what $f(D)\not\subset H_{k,\,f}$ means. We shall defer the explanation to the
end of this section.
It turns out that the proof of our first theorem opens up an approach to
other new results\,---\,with simpler, perhaps more attractive hypotheses\,---\,which we
shall present right after we state: 
                          
\begin{theorem}\label{T: mainthm2n+1}
  Let $D$ be a domain in $\mathbb{C}^m$.
  Let $\mathcal{F}$ be a family of meromorphic mappings from $D$ into $\mathbb{P}^n$,
  $\mathcal{G}$ a family of holomorphic mappings from $D$ into $\mathbb{P}^n$,   
  and $\mathcal{H}$ a collection of hyperplanes in $\mathbb{P}^n$. 
  Suppose that, for each $f\in \mathcal{F}$, there
  exist a $g_f\in\mathcal{G}$ and a subset 
  $\{H_{1,\,f},\ldots,\,H_{2n+1,\,f}\}\subset \mathcal{H}$ such that 
  \begin{enumerate}
    \item[$(a)$]  $f(D)\not\subset H_{k,\,f}$ for $k=\,1,\ldots,\,2n+1$;
  
    \item[$(b)$] $\inf\{D(H_{1,\,f},\dots, H_{2n+1,\,f}): f\in\mathcal{F}\}>0$; and
  
    \item[$(c)$] $\mathsf{supp}\,\nu(f, H_{k,\,f})\subseteq g_f^{-1}(H_{k,\,f})$
     for $k=\,1,\ldots,\,2n+1$.    
  \end{enumerate}
  Furthermore, assume that for any mapping in the closure
  of $\mathcal{G}$, its range is not a subset of
  any hyperplane in the closure of $\mathcal{H}$.
  If the family $\mathcal{G}$ is normal, then $\mathcal{F}$ is meromorphically normal.
\end{theorem}

The object $\mathsf{supp}\,\nu(f, H_{k,\,f})$ refers to the support of a certain 
non-negative divisor, whose precise definition is given in Section~\ref{S:notions}. If
$f$ were holomorphic, then the set $\mathsf{supp}\,\nu(f, H_{k,\,f})$ would
equal $f^{-1}(H_{k,\,f})$. Condition~$(c)$ should be thought of as 
``$g_f^{-1}(H_{k,\,f}) \supseteq f^{-1}(H_{k,\,f})$'', but
as $f\in \mathcal{F}$, in general, has a
non-empty indeterminacy locus, our condition needs to be stated with care. 
The collections $\mathcal{G}$ and $\mathcal{H}$ above are understood to be large. Indeed,
the larger these collections are, the less restrictive are the constraints that they impose
on $\mathcal{F}$. We shall see that the condition on the pair $(\mathcal{G}, \mathcal{H})$
towards the end of the statement of  Theorem~\ref{T: mainthm2n+1} is needed due
to $\mathcal{G}$ and $\mathcal{H}$ being large\,---\,see Theorem~\ref{T: finitefamily}
below. The reader may ask whether one could do without this condition. 
However, the condition on the pair $(\mathcal{G}, \mathcal{H})$
stated above is essential, as we show via Example~\ref{Ex: essential-sharp} below.
\smallskip

If we replace the family $\mathcal{G}$ in the above theorem by a finite collection
and let the collection of hyperplanes be discrete, then it turns out that\,---\,just following
the approach of the proof of Theorem~\ref{T: mainthm2n+1}\,---\,we may now allow
$\mathcal{G}$ to comprise meromorphic mappings. This leads to a theorem in a similar
spirit as Theorem~\ref{T: mainthm2n+1}, but with a more attractive statement\,---\,in
that it involves simpler auxiliary objects and requires fewer conditions to be checked.

\begin{theorem}\label{T: finitefamily}
  Let $D$ be a domain in $\mathbb{C}^m$.
  Let $\mathcal{F}$ be a family of meromorphic  mappings from  $D$ into $\mathbb{P}^n$ and $\mathcal{G}$
  a finite family of meromorphic mappings from $D$ into $\mathbb{P}^n$. 
  Let $\mathcal{H}$ be a discrete collection of hyperplanes in $\mathbb{P}^n$.
  Suppose that for each $f\in \mathcal{F}$, there exist a $g_f\in\mathcal{G}$ and a subset 
   $\{H_{1,\,f},\ldots,\,H_{2n+1,\,f}\}\subset \mathcal{H}$ such that 
  \begin{enumerate}
    \item[$(a)$]  $f(D)\not\subset H_{k,\,f}$ for $k=\,1,\ldots,\,2n+1$;
 
    \item[$(b)$] The hyperplanes $H_{1,\,f},\ldots,\,H_{2n+1,\,f}$ are in general position; and
 
    \item[$(c)$] $\mathsf{supp}\,\nu(f, H_{k,\,f})\subseteq \mathsf{supp}\,\nu(g_f, H_{k,\,f})$
     for $k=\,1,\ldots,\,2n+1$.  
  \end{enumerate}
  If $g(D)\not\subset H$ for every $g\in \mathcal{G}$ and $H\in \mathcal{H}$,
  then $\mathcal{F}$ is meromorphically normal.
\end{theorem}

We ought to stress that Theorem~\ref{T: finitefamily} is not merely a variant of 
Theorem~\ref{T: mainthm2n+1}. It also forms a step towards establishing a version
of the Montel--Carath{\'e}odory theorem for families of meromorphic mappings from
$D\subset \mathbb{C}^m$ into $\mathbb{P}^n$. But first, we state the following result,
which is a meromorphic analogue of
\cite[Theorem~2.8]{Yang Fang Pang 2014} by Yang--Fang--Pang:
  
\begin{corollary}\label{C: Share 2n+1}
  Let $D$ be a domain in $\mathbb{C}^m$.
  Let $\mathcal{F}$ be a family of meromorphic mappings from $D$ into $\mathbb{P}^n$, and
  let $H_1,\dots, H_{2n+1}$ be $2n+1$ hyperplanes in general position in $\mathbb{P}^n$. 
  Suppose that for each pair of mappings $f, g\in\mathcal{F}$ and each $k\in\{1,\dots, 2n+1\}$,
  $\mathsf{supp}\,\nu(f, H_k)=\mathsf{supp}\,\nu(g, H_k)$.
  Then the family $\mathcal{F}$ is meromorphically normal.
\end{corollary}

The next result is the meromorphic analogue of the Montel--Carath{\'e}odory theorem that
we had referred to at the beginning of this section. It is a corollary of the last two results.

\begin{corollary}[\textsc{Meromorphic Montel--Carath{\'e}odory Theorem}]\label{C: Mon-Cara}
  Let $D$ be a domain in $\mathbb{C}^m$.
  Let $\mathcal{F}$ be a family of meromorphic mappings from $D$ into $\mathbb{P}^n$. 
  Suppose that $H_1,\dots, H_{2n+1}$ are $2n+1$ hyperplanes in general position in $\mathbb{P}^n$. 
  If for each $k\in\{1,\dots, 2n+1\}$ and each $f\in\mathcal{F}$, 
  $\nu(f, H_k)\equiv 0$, then $\mathcal{F}$ is meromorphically normal.
\end{corollary}

Our last theorem shows that one has a version of Result~\ref{R: Yang Fang Pang 2014}
wherein one infers that a family of meromorphic mappings $\mathcal{F}$ satisfying conditions
analogous to those in Result~\ref{R: Yang Fang Pang 2014} is meromorphically normal. Moreover,
one can state such a result wherein the domain of all the maps in $\mathcal{F}$ is a domain in
$\mathbb{C}^m$ for any $m\geq 1$. Our theorem is as follows:

\begin{theorem}\label{T: mainthm3n+1}
  Let $D$ be a domain in $\mathbb{C}^m$.
 Let $\mathcal{F}$ be a family of meromorphic mappings from $D$ into $\mathbb{P}^n$,
  $\mathcal{G}$ a family of holomorphic mappings from $D$ into $\mathbb{P}^n$,   
  and $\mathcal{H}$ a collection of hyperplanes in $\mathbb{P}^n$. 
  Suppose that, for each $f\in \mathcal{F}$, there
  exist a $g_f\in\mathcal{G}$ and a subset $\{H_{1,\,f},\ldots,\,H_{3n+1,\,f}\}\subset \mathcal{H}$
  such that: 
  \begin{enumerate}
    \item[$(a)$] $f(D)\not\subset H_{k,\,f}$ and $g_f(D)\not\subset H_{k,\,f}$ for
    $k=\,1,\ldots,\,3n+1$;
  
    \item[$(b)$] $\inf\{D(H_{1,\,f},\dots, H_{3n+1,\,f}): f\in\mathcal{F}\}>0$; and
  
    \item[$(c)$] $\mathsf{supp}\,\nu(f, H_{k,\,f})\subseteq g_f^{-1}(H_{k,\,f})$
     for $k=\,1,\ldots,\,3n+1$.
  \end{enumerate}
  If the family  $\mathcal{G}$ is normal, then $\mathcal{F}$ is meromorphically normal.
\end{theorem}

It turns out\,---\,for reasons analogous to those that apply to Result~\ref{R: Yang Fang Pang 2014}\,---\,that
the number $3n+1$ in the theorem above is sharp: see Example~\ref{Ex: essential-sharp}.

The proof of Theorem~\ref{T: mainthm2n+1} will be presented in Section~\ref{S:mainthm2n+1}.
As the discussion right after the statement of this theorem suggests, there are a few basic
notion that need to be elaborated upon. This will be the focus of the next section. The
proofs of Theorem~\ref{T: finitefamily} and its corollaries will be presented in Section~\ref{S:finitefamily},
while the proof of Theorem~\ref{T: mainthm3n+1} will be presented in Section~\ref{S:mainthm3n+1}.
\smallskip

We end this section with a brief explanation of some common notations.

\subsection{Some notations}\label{SS:notations}
We fix the following notation, which we shall use without any further clarification. 
\begin{enumerate}
	\item As in the discussion above, $\|\boldsymbol{\cdot}\|$ will denote the Euclidean norm.
	Expressions like ``unit vector'' will be with reference to this norm.
	
	\item Let $f$ be a meromorphic mapping of a domain $D\subset \mathbb{C}^m$
	into $\mathbb{P}^n$ and let $I(f)$ denote the indeterminacy locus of $f$
	(which would be the empty set if $m = 1$). We write
	\[
	  \Gamma_f := \overline{\{(z, f(z)) : z\in D\setminus I(f)\}}
	\]
	(which would be precisely the graph of $f$ if it were holomorphic).
	Let $H$ be a hyperplane in $\mathbb{P}^n$.
	The notation $f(D)\not\subset H$ is our shorthand for the condition
	\[
	  \Gamma_f \text{ is not a subset of } D\times H.
       \]
\end{enumerate}

\section{Basic notions}\label{S:notions}

This section is devoted to elaborating upon concepts and terminology that made an
appearance in Section~\ref{S:intro}, and to introducing certain basic notions that we shall
need in our proofs.
\smallskip

In this section $D$ will always denote a domain in $\mathbb{C}^m$.
\smallskip

Let $f\not\equiv 0$ be a holomorphic function on $D$. 
For a point $a\in D$, let $f(z)= \sum_{j=0}^{\infty}P_j(z-a)$ denote the power-series representation of
$f$ in a neighborhood of $a$, where $P_j$ is either identically zero or a homogeneous polynomial of degree $j$.
The number 
\begin{equation*}
  \nu_{f}(a):= \min\{j\,|\,P_j\not\equiv 0\}
\end{equation*} 
is said to be the {\it zero-multiplicity of $f$ at $a$}. An integer-valued function $\nu: D\rightarrow \mathbb{Z}$ 
is called a {\it divisor} on $D$ if for each point $a\in D$ there exist holomorphic functions 
$g\not\equiv 0$ and $h\not\equiv 0$ in a neighborhood $U$ of $a$ such that 
$\nu(z)= \nu_{g}(z)-\nu_{h}(z)$ for all  $z\in U$. A divisor $\nu$ on $D$ is said to be {\it non-negative} 
if $\nu(z)\geq 0$ for all $z\in D$. We define the support 
$\mathsf{supp}\,\nu$ of a non-negative divisor $\nu$ on $D$ by
\begin{equation*}
  \mathsf{supp}\,\nu:= \{z\in D\,|\,\nu(z)\neq 0\}.
\end{equation*} 
Let $\{\nu_j\,|\,j\in\mathbb{N}\}$ be a sequence of non-negative divisors on a domain 
$D\subset\mathbb{C}^m$. We say that the sequence $\{\nu_j\}$ converges to a non-negative divisor
$\nu$ on $D$ if each $a\in D$ has a neighborhood $U$ such that there exist holomorphic functions 
$h_j(\not\equiv 0)$ and $h(\not\equiv 0)$ in $U$  with $\nu_j(z)=\nu_{h_{j}}(z)$ and $\nu(z)=\nu_h(z)$ 
for all $z\in U$, and $\{h_j\}$ converges compactly to $h$ on $U$.
\smallskip
    
The following result, due to Stoll, confirms that the support behaves continuously as a function on the space 
of non-negative divisors into the space of closed sets.

\begin{result}[paraphrasing {\cite[Theorem 4.10]{Stoll}}]\label{R:divisor-supp}
  Let $\{\nu_j\}$ be a sequence of non-negative divisors on the domain $D$.
  If $\{\nu_j\}$ converges to a non-negative divisor $\nu$, then the 
  sequence $\{\mathsf{supp}\,\nu_j\}$ converges to $\mathsf{supp}\,\nu$.
\end{result}

\noindent{We must clarify here that, given a sequence of closed sets $\{S_j\}$ of $D$, we say that
$S_j$ converges to $S$ if
\[
  \limsup_{j\to \infty}S_j := \bigcap_{k = 1}^\infty\,\bigcup_{j\geq k}S_j\,=\,\bigcup_{k = 1}^\infty\,\bigcap_{j\geq k}S_j
  =:  \liminf_{j\to \infty}S_j.
\]}

Let $S$ be an analytic set of codimension at least $2$ in $D$. By the Thullen--Remmert--Stein theorem, any 
non-negative divisor $\nu$ on $D\setminus S$ can uniquely be extended to a non-negative divisor 
$\hat{\nu}$ on $D$. Moreover, we have the following result given by Fujimoto.

\begin{result}[{\cite[page 26, (2.9)]{Fujimoto74}}]\label{R: extndivisor}
  With $D$ and $S\varsubsetneq D$ as above, if a sequence of non-negative divisors
  $\{\nu_j\,|\,j\in\mathbb{N}\}$ on $D\setminus S$ converges to 
  a divisor $\nu$ on $D\setminus S$, then $\{\hat{\nu}_j\}$ converges to $\hat{\nu}$ on $D$, where 
  $\hat{\nu}_j$ and $\hat{\nu}$ are the extensions of $\nu_j$ and $\nu$ respectively.
\end{result}

Let us now consider a meromorphic mapping $f: D\rightarrow \mathbb{P}^n$.
Fixing a system of homogeneous coordinates on $\mathbb{P}^n$, for
each $a\in D$, we have a holomorphic map 
$\widetilde{f}(z):= (f_0(z), f_1(z),\dots, f_n(z))$ on some neighborhood $U$ of $a$
such that, with $I(f)$ denoting the indeterminacy set of $f$:
\begin{itemize}
  \item $f(z) = [f_0(z): f_1(z):\cdots: f_n(z)]$ for each $z\in U\setminus I(f)$; and
  \item $I(f)\cap U = \{z\in U\,|\,f_0(z)= f_1(z)=\cdots= f_n(z)=0\}$.
\end{itemize}
We shall call any such holomorphic map $\widetilde{f} : U\rightarrow \mathbb{C}^{n+1}$
an {\it admissible representation} (or reduced representation) of $f$ on $U$. Note that
the set $\{z\in U\,|\,f_0(z)= f_1(z)=\cdots= f_n(z)=0\}$ is  
of codimension at least $2$.
\smallskip
 
Let $H$ be a hyperplane as defined in \eqref{Eq: hyperplane}, whence it is a divisor
in $\mathbb{P}^n$. Let $f: D\rightarrow \mathbb{P}^n$ be a meromorphic mapping such that
$f(D)\not\subset H$. Under this condition\,---\,see subsection~\ref{SS:notations} for
what this means\,---\,it is possible to define the
pullback of $H$ under $f$ as a {\bf divisor} in $D$, which we shall denote by $\nu(f, H)$. To
briefly see why this is so, consider 
any $a\in D$, take an admissible representation 
$\widetilde{f}:= (f_0, f_1,\dots, f_n)$ of $f$ on a neighborhood $U$ of $a$, 
and consider the holomorphic function $f^*H:= a_0f_0 + a_1f_1 +\dots+ a_nf_n$.
It follows from the definition of an admissible representation that, in a neighborhood of $a$, the values of
the divisor $\nu_{f^*H}$ do not depend on the choice of admissible representations. It is now easy
to check that if one defines $\nu(f, H)$ by
\[ 
  \left.\nu(f, H)\right|_U(z):= \nu_{f^*H}(z), \; \; z\in U,
\]
then $\nu(f, H)$ is well defined globally to give a divisor on $D$.
\smallskip

Let $\{H_j\,|\,j\in\mathbb{N}\}$ be a sequence of hyperplanes in $\mathbb{P}^n$. Then, a
{\it limiting hyperplane} of $\{H_j\}$ is the limit of any convergent
subsequence\,---\,viewing $H_j$, $j = 1, 2, 3,\dots$, as points in the dual projective
space\,---\,of $\{H_j\,|\,j\in\mathbb{N}\}$. It is useful, in view of our proofs
below, to describe a limiting hyperplane quantitatively as well. To this end: note that  
each $H_j$ has the representation 
\begin{equation*}
  H_j :=
  \left\{[w_0: w_1:\cdots: w_n]\in\mathbb{P}^n\,\big|
  \sum\nolimits_{l=0}^{n}a_{j,\,l}w_l=0\right\},
\end{equation*} 
where $(a_{j,\,0},\dots, a_{j,\,n})=:\alpha_{j}\in\mathbb{C}^{n+1}$ such that  $\|\alpha_{j}\|=1$. 
Since the sphere $\mathcal{S}:= \{z\in\mathbb{C}^{n+1}\,:\, \|z\|=1\}$ is compact, there
exist a subsequence $\{\alpha_{j_\mu}\}$ of 
$\{\alpha_{j}\}$ and $\alpha:= (a_{0},\dots, a_{n})\in \mathcal{S}$ such that
$\alpha_{j_\mu}\rightarrow \alpha$ as $\mu\rightarrow\infty$. The hyperplane
\begin{equation*}
  H :=
  \left\{[w_0 : w_1 :\cdots: w_n]\in\mathbb{P}^n\,\big| 
  \sum\nolimits_{l=0}^{n} a_{l}w_l=0\right\}.
\end{equation*}
is the limit of the sequence $\{H_{j_{\mu}}\,|\,\mu\in\mathbb{N}\}$. Conversely, any limiting
hyperplane of $\{H_j\}$ arises in this manner. (Related to the last statement is the following
observation: if $\{H_j\}$ is a convergent sequence in the dual projective space, then, since
the vectors $\alpha_j$ associated to $H_j$, $j = 1, 2, 3,\dots$, as described above are not uniquely
determined, the auxiliary sequence $\{\alpha_j\}$ need not be convergent. However, if
$\{H_j\}$ is convergent, then each subsequential limit of $\{\alpha_j\}$ would determine the same
hyperplane.)
Let $\mathcal{H}$ be a collection of hyperplanes in $\mathbb{P}^n$. 
Given the structure of the space of all hyperplanes in $\mathbb{P}^n$,
the closure of $\mathcal{H}$\,---\,which appears in the statement of
Theorem~\ref{T: mainthm2n+1}\,---\,is just the union of $\mathcal{H}$ and the
set of all limiting hyperplanes of $\mathcal{H}$.
\smallskip

Let $M$ be a compact connected Hermitian manifold.
The space $\operatorname{Hol}(D, M)$ of holomorphic 
mappings from $D$ into $M$ is endowed with the compact-open topology.
\begin{definition}\label{Def:normal}
  A family $\mathcal{F}\subset\operatorname{Hol}(D, M)$ is said to be {\it normal} if 
  $\mathcal{F}$ is relatively compact in $\operatorname{Hol}(D, M)$.
\end{definition}
 
\begin{definition}[{\cite[Definition 3.1]{Fujimoto74}}]\label{Def:mcgt}
  Let $\{f_j\,|\,j\in\mathbb{N}\}$ be a sequence of meromorphic mappings from $D$ into
  $\mathbb{P}^n$. The sequence $\{f_j\}$ is said to {\it converge meromorphically}  
  on $D$ to a meromorphic mapping $f$ if, for
  each $a\in D$, there exists an open neighborhood $U$ of $a$ and an admissible representation
  \begin{equation*}
    \widetilde{f}_j(z) :=
    \big(f_{j,\,0}(z), f_{j,\,1}(z),\dots, f_{j,\,n}(z)\big)
  \end{equation*}
  of $f_j$ on $U$, $j = 1, 2, 3,\dots$, such that, for each $l\in\{0, 1,\dots, n\}$, 
  the sequence $\{f_{j,\,l}\}$ converges uniformly on compact subsets of $U$ to a 
  holomorphic function $f_l$ on $U$ with the property that
  \begin{equation*}
    f(z) := [f_0(z) : f_1(z) :\cdots : f_n(z)] \; \;
    \forall z\in U\setminus\bigcap\nolimits_{l = 0}^nf_l^{-1}\{0\},
  \end{equation*}
  where $f_{l_0}(z)\not\equiv 0$ on $U$ for some $l_0$.
\end{definition}

We now have all the terminology needed for the definition that is central to the discussion
in Section~\ref{S:intro}.

\begin{definition}[{\cite[Definition 4.1]{Fujimoto74}}]\label{Def:mnormal}
  A family $\mathcal{F}$ of meromorphic mappings from $D$ into $\mathbb{P}^n$ is said to be a 
  {\it meromorphically normal family} if any sequence in $\mathcal{F}$ has a meromorphically 
  convergent subsequence on $D$.
\end{definition}

\section{Some Examples}\label{S: Examples}

We now provide the examples alluded to in Section~\ref{S:intro}.

\begin{example}\label{Ex: cond(c)}
  The condition $(c)$ in the statement of Theorem~\ref{T: mainthm2n+1} is essential.
\end{example}
  
\noindent{Let $D$ be the open unit disc in $\mathbb{C}$, and $\mathcal{F}_1 = \{f_j(z)\,|\,j\in\mathbb{N}\}$,
where $f_j: D\rightarrow \mathbb{P}^1$ is defined by
\begin{equation*}
  f_j(z):= [\cos jz : \sin jz].
\end{equation*}
Let $\mathcal{G}_1$ be the singleton consisting of the map
$D\ni z\longmapsto[1 : z/2 ]$. Call this map $g$.
Let 
\begin{align*}
  H_1 & :=  \{[w_0: w_1]\,|\, w_0 - iw_1 = 0\}, \\
  H_2 & :=  \{[w_0: w_1]\,|\, w_0 + iw_1 = 0\}, \text{ and} \\
  H_3 & :=  \{[w_0: w_1]\,|\, w_1=0\}.
\end{align*}
These hyperplanes are in general position. The condition~$(c)$ in the statement of
Theorem~\ref{T: mainthm2n+1} fails  because
$g^{-1}(H_3) = \{0\}$ whereas $f_j^{-1}(H_3)$ has a large number of points
for $j$ large. Thus, all the conditions in the statement of
Theorem~\ref{T: mainthm2n+1} hold true except for $(c)$ (here
$\mathsf{supp}\,\nu(f_j, H_3) = f_j^{-1}(H_3)$).  
However, the family $\mathcal{F}_1$ fails to be meromorphically normal.
\hfill $\blacktriangleleft$}  
\smallskip

\begin{example}
  The number $2n+1$ is sharp in the statement of Theorem~\ref{T: mainthm2n+1}.
\end{example}

\noindent{Let $D$ be the open unit disc in $\mathbb{C}$, and $\mathcal{F}_2= \{f_j(z)\,|\,j\in\mathbb{N}\}$,
where $f_j : D\rightarrow \mathbb{P}^1$ is defined by
\begin{equation*}
  f_j(z):= [1 : e^{jz}].
\end{equation*}
Let $\mathcal{G}_2$ be the singleton consisting of the map ${\sf id}_D$.
Let $H_k:= \{[w_0: w_1]\,|\,w_k = 0\}$, $k=0, 1$, be two hyperplanes in $\mathbb{P}^1$.
These hyperplanes are in general position. Clearly, $\mathcal{F}_2,\, \mathcal{G}_2$,
and $\{H_1, H_2\}$ satisfy all other conditions\,---\,with the understanding that $k$ is limited
to $k = 1, 2$\,$(=2n)$\,---\,in the statement of Theorem~\ref{T: mainthm2n+1}. However, 
$\mathcal{F}_2$ is not meromorphically normal. \hfill $\blacktriangleleft$} 
\smallskip
  
The following example illustrates that we cannot eliminate the condition imposed on the pair
$(\mathcal{G}, \mathcal{H})$. Observe: this example also confirms that the number $3n+1$ in the statement
of Theorem~\ref{T: mainthm3n+1} is also sharp. This example is taken from \cite{Yang Fang Pang 2014},
where it serves a different purpose (somewhat resembling our previous observation). We repurpose that
example as follows:
     
\begin{example}[paraphrasing {\cite[Example~1]{Yang Fang Pang 2014}}]\label{Ex: essential-sharp}
  The condition imposed on the pair $(\mathcal{G}, \mathcal{H})$ in the statement of
  Theorem~\ref{T: mainthm2n+1} is essential. Specifically: the conclusion of
  Theorem~\ref{T: mainthm2n+1} need not follow if the range of some limit
  point of $\mathcal{G}$ is contained in some hyperplane in the closure of $\mathcal{H}$.
  Additionally, the number $3n+1$ in the statement of Theorem~\ref{T: mainthm3n+1} is also sharp.
\end{example}

{\noindent Let $D$ be the open unit disc in $\mathbb{C}$, and 
$\mathcal{F}_3= \{f_j(z)\,|\,j\in\mathbb{N}\}$, where $f_j: D\rightarrow \mathbb{P}^2$ 
is defined by
\begin{equation*}
  f_j(z):= [i\cos jz: \sin jz: \sin jz].
\end{equation*} 
Let us denote by $\alpha_{j,\,1}, \alpha_{j,\,2},\dots, \alpha_{j,\,k_j}$ the zeros of $\sin jz$ in $D$. Let 
$\mathcal{G}_3= \{g_j\,|\,j\in\mathbb{N}\}$, where $g_j: D\rightarrow \mathbb{P}^2$ 
is defined by
\begin{equation*}
  g_j(z):= \left[1: \prod\nolimits_{1\leq t\leq k_j}\frac{z-\alpha_{j,\,t}}{1-\bar{\alpha}_{j,\,t}z}: 
  \prod\nolimits_{1\leq t\leq k_j}\frac{z-\alpha_{j,\,t}}{1-\bar{\alpha}_{j,\,t}z}\right].
\end{equation*}
Let 
\begin{align*}
  H_{1} & := \{[w_0: w_1: w_2]\,|\, 3w_0 + w_1 + 2w_2 = 0\}, \\
  H_{2} & := \{[w_0: w_1: w_2]\,|\, -5w_0 + w_1 + 4w_2 = 0\}, \\
  H_{3} & := \{[w_0: w_1: w_2]\,|\, 7w_0 + w_1 + 6w_2 = 0\}, \\
  H_{4} & := \{[w_0: w_1: w_2]\,|\, -9w_0 + w_1 + 8w_2 = 0\}, \\
  H_{5} & := \{[w_0: w_1: w_2]\,|\, w_2 = 0\}, \text{ and } \\
  H_{6} & := \{[w_0: w_1: w_2]\,|\, w_1 = 0\}.       
\end{align*}  
 It is easy to see that $D(H_{1},\dots, H_{6})>0$.
Also, we have $g_j(D)\not\subset H_{k}$ for all $j$ and $k=1,\dots, 6$. 
Note that $g_j(z)\rightarrow g(z)$, where $g$ is the map $D\ni z\mapsto [1: 0: 0]$, this 
clearly shows that $g(D)\subset H_5$. It is not hard to see that
$f_j^{-1}(H_{k})= g_j^{-1}(H_{k})$
for all $j$ and $k=1,\dots, 6$. Thus, all the conditions in the statement of: 
\begin{itemize}
  \item Theorem~\ref{T: mainthm2n+1} hold true except for the condition imposed on the pair 
  $(\mathcal{G}, \mathcal{H})$.
  
  \item Theorem~\ref{T: mainthm3n+1} hold true, with the understanding that $k$ is limited
  to $k = 1,\dots, 6$\,$(=3n)$.
\end{itemize}
However, $\mathcal{F}_3$ fails to be meromorphically normal. \hfill $\blacktriangleleft$} 
\medskip

\section{Essential lemmas}\label{S: lemmas}

In order to prove our theorems, we need to state certain known results.
\smallskip

One of the well-known tools in the theory of normal families in one complex variable is Zalcman's lemma.
Roughly speaking, it says that the failure of normality implies that a certain kind of infinitesimal 
convergence must take place. The higher dimensional analogue of Zalcman's rescaling lemma 
is as follows:

\begin{lemma}[{\cite[Theorem 3.1]{AladroKrantz}}]\label{L: AK}
  Let $M$ be a compact complex space, and $\mathcal{F}$ a family of holomorphic mappings 
  from a domain $D\subset \mathbb{C}^m$ into $M$. The family $\mathcal{F}$ is not normal 
  if and only if there exist
  \begin{enumerate}
    \item[$(a)$] a point $\xi_0\in D$ and a sequence $\{\xi_j\}\subset D$ such that  $\xi_j\rightarrow \xi_0$;
    
    \item[$(b)$] a sequence $\{f_j\}\subset \mathcal{F}$;
    
    \item[$(c)$] a sequence $\{r_j\}\subset \mathbb{R}$ with $r_j>0$ and $r_j\rightarrow 0$; and
    
    \item[$(d)$] a sequence $\{u_j\}$ of unit vectors in $\mathbb{C}^m$
  \end{enumerate} 
  such that $ h_j(\zeta):= f_j(\xi_{j}+r_{j}u_{j}\zeta)$, where $\zeta\in\mathbb{C}$ satisfies
  $\xi_{j}+r_{j}u_{j}\zeta\in D$, converges uniformly on compact subsets of $\mathbb{C}$ to a 
  non-constant holomorphic mapping $h: \mathbb{C}\rightarrow M.$
\end{lemma}

\begin{remark}
  We remark that the result of Aladro--Krantz in \cite{AladroKrantz} has a weaker hypothesis
  than in Lemma~\ref{L: AK}. In the case where $M$ is {\it non-compact}, there is a case missing from
  their analysis. The arguments needed in this case were provided by
  \cite[Theorem 2.5]{Thai Trang Huong 03}. At any rate, Lemma~\ref{L: AK} is the version of the
  Aladro--Krantz theorem that we need.
\end{remark} 

The following result, due to Fujimoto, is about the extension of the domain of 
meromorphic convergence of a certain sequence of meromorphic mappings.
 
\begin{lemma}[{\cite[Proposition 3.5]{Fujimoto74}}]\label{L: Fujimotomeronormalprop}
  Let $D$ be a domain in $\mathbb{C}^m$, and $S$ a proper analytic subset of $D$. 
  Let $\{f_j\}$ be a sequence of meromorphic mappings from $D$ into $\mathbb{P}^n$. 
  Suppose that $\{f_j\}$ converges meromorphically to a meromorphic mapping $f$ on $D\setminus S$. 
  If there exists a hyperplane $H$ in $\mathbb{P}^n$ such that $f(D\setminus S)\not\subset H$ and 
  the sequence $\{\nu(f_j, H)\}$ of divisors converges on $D$, then $\{f_j\}$ converges meromorphically on $D$.
\end{lemma}

Eremenko gave the following interesting result wherein every holomorphic mapping $f$ from the complex plane
$\mathbb{C}$ into a projective variety in $\mathbb{P}^n$ becomes a constant mapping 
provided $f$ omits a finite number of certain hypersurfaces. 
 
\begin{lemma}[{\cite[Theorem 1]{Eremenko99}}]\label{L: Eremenko-hyperbolic}
  Let $X\subset\mathbb{P}^n$ be a projective variety, and $N$ a positive integer.
  Let $H_1,\dots, H_{2N+1}$ be $2N+1$ hypersurfaces in $\mathbb{P}^n$ with the property
  \begin{equation*}
    X\cap\Big(\bigcap\nolimits_{k\in I}H_k\Big)= \emptyset \,\,\,
    \text{ for every } I\subset \{1,\dots, 2N+1\} \text{ such that } |I| =N+1,
  \end{equation*} 
  where $|I|$ is the cardinality of the set $I$. Then, every holomorphic mapping from 
  $ \mathbb{C}$ into $ X\setminus\big(\cup_{k=1}^{2N+1} H_k\big)$ is constant. 
\end{lemma}

\section{The proof of Theorem~\ref{T: mainthm2n+1}}\label{S:mainthm2n+1}

Certain parts of our proof of Theorem~\ref{T: mainthm2n+1} will rely on 
techniques similar to those in \cite{Dethloff,Quang Tan08}. Since our proof will, at a certain
stage, rely upon results involving the convergence of divisors, we shall rephrase Condition~$(c)$
in a form that involves the supports of the divisors $g_f^{-1}(H_k)$, $k = 1,\dots, 2n+1$.
With these words, we give the

\begin{proof}[Proof of Theorem\,\ref{T: mainthm2n+1}]

Let $\{f_j\}\subset\mathcal{F}$  be  an arbitrary sequence. By the hypothesis of the theorem,
 there exist:
\begin{itemize}
  \item { a sequence $\{g_j\}\subset\mathcal{G}$;}

  \item { $2n+1$ sequences of 
  hyperplanes $\left\{H_{k,\,j}\right\}$, $k=1,\ldots, 2n+1$; and}
  
  \item { $2n+1$ hyperplanes $H_1,\dots, H_{2n+1}$;}
\end{itemize}
such that for each $j$ and each $k$
\begin{equation*}
 f_j(D)\not\subset H_{k,\,j} \text{ and } g_j(D)\not\subset H_{k,\,j}\,; \; \;
 D(H_{1,\,j},\ldots,\,H_{2n+1,\,j})> 0;
 \text{ and }
\end{equation*} 
\begin{equation}\label{Eq: support}
  \mathsf{supp}\,\nu(f_j,\,H_{k,\,j})\subseteq \mathsf{supp}\,\nu(g_j,\,H_{k,\,j});
\end{equation}
and such that the hyperplane $H_k$ is a limiting hyperplane of
$\left\{H_{k,\,j}\right\}$, $k=1,\ldots, 2n+1$. To be more precise: it follows
from the discussion in Section~\ref{S:notions} that there exists
an increasing sequence $\{j_{\mu}\}\subset \mathbb{N}$
such that
$H_k$ is the limit of $\left\{H_{k,\,j_{\mu}}\right\}$, $k = 1,\dots, 2n+1$. 
We also record the following, which will be relevant to the proof of Theorem~\ref{T: finitefamily}:
\begin{itemize}
  \item[$(*)$] If $\mathcal{H}$ is a discrete collection of hyperplanes, then
  the sequences $\{H_{k,\,j_{\mu}}\}$ are constant subsequences.
\end{itemize}
Owing to the condition $(b)$, the hyperplanes $H_1,\ldots, H_{2n+1}$ are in general
position in $\mathbb{P}^n$. This follows from the definition of $D(H_1,\dots, H_{2n+1})$
and the quantitative discussion on limiting hyperplanes, of a sequence of hyperplanes, 
in Section~\ref{S:notions}.
\smallskip

Since the family $\mathcal{G}$ is normal, there exists a subsequence of $\{g_{j_\mu}\}$ that
converges compactly to a holomorphic mapping $g$.
By hypothesis,
$g(D)\not\subset H_k,\ k=1,\ldots, 2n+1$.
Set 
\begin{equation*}
S_k := \mathsf{supp}\,\nu(g,\,H_k),\ k=1,\dots, 2n+1,
\end{equation*}
and
\begin{equation*}
E:=\bigcup_{k=1}^{2n+1} S_k.
\end{equation*} 
By Result~\ref{R:divisor-supp}, $E$ is either an empty set or an analytic subset of  codimension 1 in $D$.
\smallskip

Fix a point $z_0\in D\setminus E$ and 
choose a relatively compact open neighborhood $U_{z_0}$ of
$z_0$ in $D\setminus E$.
We now pass to that subsequence of $\{f_j\}$ that is indexed by the integer-sequence that indexes
the subsequence of $\{g_{j_\mu}\}$ introduced in the previous paragraph. At this stage, we may,
without loss of generality, relabel the two sequences referred to as $\{g_j\}$ and $\{f_j\}$. With
this relabelling, we conclude\,---\,in view of \eqref{Eq: support} and the fact that
$U_{z_0}\cap S_k=\emptyset$\,---\,that $U_{z_0}\cap\mathsf{supp}\,\nu\left(f_j,\,H_{k,\,j}\right)
=\emptyset$ for every $j$ and every $k$. Hence 
$\big\{\!\left.f_j\right|_{U_{z_0}}\!\big\}\subset \operatorname{Hol}\left(U_{z_0},\,\mathbb{P}^n\right) $.
\smallskip

We shall now prove that the sequence $\big\{\!\left.f_j\right|_{U_{z_0}}\!\big\}$ has a subsequence
that converges compactly on $U_{z_0}$. Let us assume that the latter is
not true, and aim for a contradiction. Then, by Lemma~\ref{L: AK}, there exist
\begin{enumerate}
  \item[(i)] a subsequence of $\big\{\!\left.f_j\right|_{U_{z_0}}\!\big\}$, which we may label
  without causing confusion\,---\,{\bf just} for this paragraph\,---\,as
  $\big\{\!\left.f_j\right|_{U_{z_0}}\!\big\}$;
  
  \item[(ii)] a point $\xi_0\in U_{z_0}$ and a
  sequence $\{\xi_j\}\subset U_{z_0}$ such that $\xi_j\rightarrow \xi_0$;

\item[(iii)] a sequence  $\{r_j\}\subset \mathbb{R}$ with $r_j>0$ such that $r_j\rightarrow0$; and

\item[(iv)]  a sequence  $\{u_j\}$ of unit vectors in $\mathbb{C}^m$
\end{enumerate}
such that\,---\,defining the maps $h_j: \zeta\mapsto f_j(\xi_j+r_j u_j \zeta)$ on suitable
neighborhoods of $0\in \mathbb{C}$\,---\,$\{h_j\}$
converges uniformly on compact subsets of $\mathbb{C}$ to a non-constant
holomorphic mapping $h:\mathbb{C}\rightarrow\mathbb{P}^n$. Then, there exist admissible representations
\begin{equation*} \widetilde{h}_j = \big(h_{j,\,0}, h_{j,\,1},\ldots, h_{j,\,n}\big)  \text{ and }
  \widetilde{h} = \big(h_0, h_1,\ldots, h_n\big)
\end{equation*}
of $h_j$ and $h$ respectively such that $\left\{h_{j,\,l}\right\}$ converge uniformly on compact subsets
of $\mathbb{C}$ to $h_l$, $0\leq l\leq n$.
This implies that $\big\{{h}_j^*H_{k,\,j}\big\}$ converge
uniformly on compact subsets of $\mathbb{C}$ to $h^*H_k$,
$1\leq k\leq 2n+1$. Recall that we defined $f^*H$, where $H$ is any hyperplane,
in Section~\ref{S:notions}.
By Hurwitz's theorem, and the
fact that for each $k\in\{1,\ldots, 2n+1\},$ 
$U_{z_0}\cap\mathsf{supp}\,\nu\left(f_j,\,H_{k,\,j}\right)=\emptyset$
for every $j$, one of the following holds for each $k$:
\begin{enumerate}
  \item $h(\mathbb{C})\subset H_k$, or
    
  \item $h(\mathbb{C})\cap H_k=\emptyset$.
\end{enumerate}
Now, let $J:=\left\{k\in\{1,\ldots, 2n+1\}\,\big|\,h(\mathbb{C})\subset H_k\right\}$, and
\begin{equation*} 
  Z:=\begin{cases}
         \bigcap_{k\in\,J}\!H_k, &\text{if $J\neq\emptyset$}, \\
         \mathbb{P}^n, &\text{if $J=\emptyset$}.
       \end{cases}
\end{equation*}
We now make the following {\bf claim:}
\begin{itemize}
  \item $|J^c|\geq 2\dim_{\mathbb{C}}Z + 1$; and
  \item for each $I\subset J^c$ with $|I|
  = \dim_{\mathbb{C}}Z + 1$, $Z\cap\left( \bigcap_{k\in\,I}H_k\right)=\emptyset$.
\end{itemize}
If $J=\emptyset$, then we are done because in that case $Z=\mathbb{P}^n$, and 
the hyperplanes $H_1,\ldots,\,H_{2n+1}$ are in general position. Now, we consider the case
when $J\neq\emptyset$. Since the hyperplanes 
$H_1,\ldots,\,H_{2n+1} $ are in general position, we have:

\begin{itemize}
  \item[{}] For each $1\leq t\leq 2n+1$ and $I\subset \{1,\ldots, 2n+1\}$ with $|I|=t$,
  $\bigcap_{k\in\,I}\!H_k$ is of pure dimension equal to $\dim_{\mathbb{C}}\big(\bigcap_{k\in\,I}\!H_k\big)
  =\max\{n-t,\,-1\}$ (here, it is understood that
  $\dim_{\mathbb{C}}(\emptyset)=-1)$.
\end{itemize}
The above follows from B{\'e}zout's theorem. Thus, we get
\begin{equation*}
  \dim_{\mathbb{C}}Z =\dim_{\mathbb{C}}\Big(\bigcap\nolimits_{k\in\,J}H_k\Big) = \max\{n-|J|,\,-1\}.
\end{equation*}
Since $h(\mathbb{C})\subset Z$, $\dim_{\mathbb{C}}Z\geq 0$, whence we get
\begin{equation}\label{Eq: |J|}
  |J|=n-\dim_{\mathbb{C}}Z.
\end{equation}
Hence,
\begin{equation*}
  |J^c| = n+\dim_{\mathbb{C}}Z+1 \geq 2\dim_{\mathbb{C}}Z+1.
\end{equation*}
The inequality above is, again, a consequence of $H_1,\ldots,\,H_{2n+1} $ being in general position.
Moreover, if $I\subset J^c$ with $|I|=\operatorname{dim}_{\mathbb{C}}Z+1$, then by \eqref{Eq: |J|},
\begin{equation*}
  |I\cup J|=(\dim_{\mathbb{C}}Z+1)+(n-\dim_{\mathbb{C}}Z)=n+1,
\end{equation*}                                                                            
and
\begin{equation*}
   Z\cap\Big(\bigcap\nolimits_{k\in\,I}H_k\Big) = \Big(\bigcap\nolimits_{k\in\,J}H_k\Big)
   \cap \Big(\bigcap\nolimits_{k\in\,I}H_k\Big) = \bigcap\nolimits_{k\in\,I\,\cup\,J}H_k = \emptyset,
\end{equation*}
where the last equality holds because the hyperplanes $H_1,\ldots, H_{2n+1}$ are 
in general position. This establishes the above claim.
At this stage, we can appeal to Lemma~\ref{L: Eremenko-hyperbolic}, to conclude that $h$ is constant, which is a
contradiction. Thus, it follows that the sequence $\big\{\!\left.f_j\right|_{U_{z_0}}\!\big\}$\,---\,by which
we now mean the sequence considered at the beginning of this paragraph\,---\,has a subsequence
that converges compactly on $U_{z_0}$. Let us denote the limit of this subsequence by 
$f^{U_{z_0}}$.
\smallskip

Cover $D\setminus E$ by a countable collection of connected open subsets of $D\setminus E$.
Note that the point $z_0\in D\setminus E$ and the neighborhood $U_{z_0}\ni z_0$ in
the last paragraph were arbitrary (with the proviso $U_{z_0}\cap E = \emptyset$). Thus, from
the conclusion of the last paragraph and by a standard diagonal argument, we conclude that
there exists a subsequence of the {\bf original} sequence $\{f_j\}$\,---\,which we can again denote
simply by $\{f_j\}$\,---\,and a holomorphic map $f$ on $D\setminus E$ such that $\{f_j\}$
converges compactly on $D\setminus E$ to $f$.
\smallskip 
  
We shall now show that there is a subsequence of $\left\{f_j\right\}$ 
that is meromorphically convergent in $D$. 
Since the hyperplanes $H_1,\ldots, H_{2n+1}$ are in general position in $\mathbb{P}^n$, there exists a 
subset  $L\subset\{1,\ldots, 2n+1\}$ with $|L|\geq n+1$ such that $f^*H_k\not\equiv0$ 
on $D\setminus E$ for all $k\in L$. 
Let $z_1\in D$ be an arbitrary point. By the general-position condition, again, there
exist an open ball $B(z_1,\epsilon)\subset D$
and $k_0\in L$ such that 
$\mathsf{supp}\,\nu(g,\,H_{k_0})\cap\overline{B(z_1,\epsilon)}=\emptyset$. This implies
that by passing to appropriate subsequences once more\,---\,and relabelling them by 
$j\in \mathbb{N}$\,---\,we have 
subsequences $\{f_j\}$, $\{g_j\}$ and $\{H_{k_0,\,j}\}$ such that, firstly,
$\mathsf{supp}\,\nu(g_j,\,H_{k_0,\,j})\cap\overline{B(z_1,\epsilon)}=\emptyset.$
Then, by \eqref{Eq: support}, we have 
\begin{equation}\label{Eq: suppfjempty}
  \mathsf{supp}\,\nu(f_j,\,H_{k_0,\,j})\cap\overline{B(z_1,\epsilon)}=\emptyset.
\end{equation}

Now, we define meromorphic mappings $F_j: B(z_1,\epsilon)\rightarrow\mathbb{P}^{n+1}$, $j\in\mathbb{N}$, 
as follows: for any $z\in B(z_1,\epsilon)$, if
$\widetilde{f}_j=(f_{j,\,0}, f_{j,\,1},\ldots, f_{j,\,n})$ is an admissible representation
of $f_j$ on a neighborhood
$U_z\subset B(z_1,\epsilon)$, then $F_j$ is such that it has the admissible
representation $\widetilde{F}_j= (f_{j,\,0}, f_{j,\,1},\dots, f_{j,\,n}, f_j^*H_{k_0,\,j})$ 
on $U_z$. Let $Q_l$, $l=0, 1,\dots, n$, be
hyperplanes in $\mathbb{P}^n$ defined by
\begin{equation*}
  Q_l := \big\{[w_0: w_1:\cdots: w_{n}]\in\mathbb{P}^n\,\big|\,w_l=0\big\},
\end{equation*}
and let $\overline{Q}_l$, $l=0, 1,\dots, n+1$, be hyperplanes in $\mathbb{P}^{n+1}$
defined by
\begin{equation*}
  \overline{Q}_l := \big\{[w_0: w_1:\cdots: w_{n+1}]\in\mathbb{P}^{n+1}\,\big|\,w_l=0\big\}.
\end{equation*}
Clearly, $\{F_j\}$ converges compactly on $B(z_1,\epsilon)\setminus E$ to a holomorphic
mapping $F$ from $B(z_1,\epsilon)\setminus E$ into
$\mathbb{P}^{n+1}$, and if $\widetilde{f}= \big(f_0,\dots, f_n\big)$ is an admissible representation
of $f$ on an open subset $U\subset B(z_1,\epsilon)\setminus E$, then $F$ has an admissible
representation $\widetilde{F}= (f_0,\,f_1,\dots, f_n, f^*H_{k_0})$ on $U$. Since $f$
is holomorphic on $B(z_1,\epsilon)\setminus E$, there exists  $l_0$, $0\leq l_0\leq n$ such
that $f^*Q_{l_0}\not\equiv0$ on
$B(z_1,\epsilon)\setminus E$. Hence, $F^*\overline{Q}_{l_0}\not\equiv0$ on
$B(z_1,\epsilon)\setminus E.$  
Therefore, there exists a $j_0\in \mathbb{N}$
such that $f_j^*Q_{l_0}\not\equiv 0$ and $F_j^*\overline{Q}_{l_0}\not\equiv 0$ on
$B(z_1,\epsilon)\setminus E$ for all $j\geq j_0$.
Since $f^*H_{k_0}\not\equiv 0$ on $B(z_1,\epsilon)\setminus E$, we have $F^*\overline{Q}_{n+1}\not\equiv 0$
on $B(z_1,\epsilon)\setminus E.$ Also, by \eqref{Eq: suppfjempty}, we have $\nu(F_j, \overline{Q}_{n+1})=0.$ 
Therefore by Lemma \ref{L: Fujimotomeronormalprop}, $\{F_j\}$  
converges meromorphically  on $B(z_1,\epsilon)$. This implies that the sequence of divisors 
$\big\{\nu(F_j,\,\overline{Q}_{l_0})\big\}_{j\geq\,j_0}$ converges on $B(z_1,\epsilon),$ hence 
$\big\{\nu(f_j,\, Q_{l_0})\big\}_{j\geq\,j_0}$ converges on $B(z_1, \epsilon).$ By 
Lemma~\ref{L: Fujimotomeronormalprop} again, $\{f_j\}_{j\geq\,j_0}$ is meromorphically 
convergent on $B(z_1, \epsilon).$ By a diagonal argument\,---\,with details analogous to the
argument made above\,---\,we extract a further subsequence that is meromorphically convergent on $D$ to a
meromorphic mapping 
which agrees with $f$ on $D\setminus E$. This completes the proof.
\end{proof}

\section{The proofs of Theorem~\ref{T: finitefamily} and its corollaries}\label{S:finitefamily}

\begin{proof}[Proof of Theorem~\ref{T: finitefamily}]
Let $\{f_j\}\subset\mathcal{F}$  be  an arbitrary sequence. We can find
$2n+1$ sequences of hyperplanes $\left\{H_{k,\,j}\right\}$, $k=1,\ldots, 2n+1$, associated
to $\{f_j\}$ exactly as in the beginning of the proof of Theorem~\ref{T: mainthm2n+1}.
The purpose of Condition~$(b)$ of Theorem~\ref{T: mainthm2n+1} was to ensure that the
hyperplanes $H_1,\dots, H_{2n+1}$ presented right at the beginning of the proof of
Theorem~\ref{T: mainthm2n+1} are in general position. As highlighted by $(*)$ in
the proof of Theorem~\ref{T: mainthm2n+1}: in our present setting
(i.e., $\mathcal{H}$ is discrete) $H_k$ is just the single hyperplane 
repeated {\it ad infinitum} that constitutes the constant sequence
$\{H_{k,\,j_{\mu}}\}$ referred to by $(*)$, $k = 1,\dots, 2n+1$.
Now, $H_1,\dots, H_{2n+1}$ being 
in general position follows from $(b)$ in the statement of Theorem~\ref{T: finitefamily}.
\smallskip

Since the family $\mathcal{G}$ is a finite family, we can find a meromorphic mapping $g\in \mathcal{G}$
such that we can extract a subsequence of $\{f_j\}$, which we can continue to label as
$\{f_j\}$ without loss of generality, such that
\begin{equation}\label{Eq: support 2}
  \mathsf{supp}\,\nu(f_j,\,H_{k,\,j})\subseteq \mathsf{supp}\,\nu(g,\,H_{k,\,j})
\end{equation}
for each $j$ and each $k$.
By hypothesis, $g(D)\not\subset H_k$ for all $k\in\{1,\dots, 2n+1\}$.
Now, define the sets $S_k$ and $E$ as they were defined in 
the proof of Theorem~\ref{T: mainthm2n+1}.
Thus, we have all of the ingredients to be able to repeat verbatim the argument beginning from the third
paragraph of the proof of Theorem~\ref{T: mainthm2n+1} until the fifth paragraph thereof
to extract from $\{f_j\}\subset \mathcal{F}$ a subsequence\,---\,which we may continue, without loss
of generality, to refer to as $\{f_j\}$\,---\,that converges compactly on $D\setminus E$ to a map $f$
that is holomorphic on $D\setminus E$.
\smallskip

Since the hyperplanes $H_1,\ldots, H_{2n+1}$ are in general position in $\mathbb{P}^n$, there exists a 
subset  $L\subset\{1,\ldots, 2n+1\}$ with $|L|\geq n+1$ such that $f^*H_k\not\equiv0$ 
on $D\setminus E$ for all $k\in L$. 
\smallskip

We shall now show that there is a subsequence of $\left\{f_j\right\}$ 
that is meromorphically convergent in $D\setminus \bigcap_{k\in\,L} S_k$. 
Let $z_1\in D\setminus\bigcap_{k\in\,L} S_k$ be an arbitrary point. There
exist an open ball $B(z_1,\epsilon)\subset D\setminus\bigcap_{k\in\,L} S_k$
and $k_0\in L$ such that $S_{k_0}\cap\overline{B(z_1,\epsilon)}=\emptyset.$ Hence, 
$\mathsf{supp}\,\nu(g,\,H_{k_0})\cap\overline{B(z_1,\epsilon)}=\emptyset$. This implies
that by passing to appropriate subsequences once more\,---\,and relabelling them by 
$j\in \mathbb{N}$\,---\,we have 
subsequences $\{f_j\}$, $\{g_j\}$ and $\{H_{k_0,\,j}\}$ such that, firstly,
$\mathsf{supp}\,\nu(g_j,\,H_{k_0,\,j})\cap\overline{B(z_1,\epsilon)}=\emptyset.$
Then, by \eqref{Eq: support 2}, we have 
\begin{equation}\label{Eq: suppfjempty 2}
  \mathsf{supp}\,\nu(f_j,\,H_{k_0,\,j})\cap\overline{B(z_1,\epsilon)}=\emptyset.
\end{equation}

We are again in a position to repeat the argument in the
final paragraph of the proof of Theorem~\ref{T: mainthm2n+1}, {\it mutatis mutandis}, to conclude
that $\{f_j\}$ has a subsequence\,---\,which we may again relabel as $\{f_j\}$\,---\,that
is meromorphically convergent on $D\setminus \bigcap_{k\in\,L}S_k$ to a
meromorphic mapping $\varphi$ 
(which agrees with $f$ on $D\setminus E$). 
\smallskip

We now claim that the indeterminacy set $I(g)=\bigcap_{k\in\,L} S_k$. We remark here
that this claim presents {\bf no} contradiction when $m=1$, in which
case $g$ is holomorphic. This is because, by holomorphicity of $g$, we get
\[
  S_k = g^{-1}(H_k) \quad\text{and}
  \quad \bigcap\nolimits_{k\in L}S_k = \emptyset.
\]
The latter is a consequence of the general-position condition, which here translates into
$\{H_k : k\in L\}$ being a set of distinct points. At any rate, the following
argument is not much different and holds for any $m$. Clearly,
$I(g)\subset\bigcap_{k\in\,L} S_k.$ Since 
$|L|\geq n+1$, and $H_k$, $k\in L$, are in general position in $\mathbb{P}^n$, we have
$\bigcap_{k\in\,L} S_k\subset I(g)$\,---\,hence the claim. Now, the meromorphic normality of
$\mathcal{F}$ on $D$ follows from the following argument:
Take a hyperplane $H$ in $\mathbb{P}^n$ such that $\varphi(D\setminus\bigcap_{k\in\,L}S_k)\not\subset H$. Clearly,
$\left\{\nu(f_j,\,H)\right\}$ converges on $D\setminus\bigcap_{k\in\,L}S_k$. Since the codimension
of the set $\bigcap_{k\in\,L}S_k=I(g)$ is at least 2, we have by
Result~\ref{R: extndivisor} that $\left\{\nu(f_j,\,H)\right\}$ converges on
$D$. Therefore by Lemma~\ref{L: Fujimotomeronormalprop}, $\left\{f_j\right\}$ converges
meromorphically on $D$. This completes the proof.
\end{proof}

The proofs of Corollaries~\ref{C: Share 2n+1}~and~\ref{C: Mon-Cara} are absolutely direct
applications. For the sake of completeness, we now provide their proofs: 

\begin{proof}[Proof of Corollary~\ref{C: Share 2n+1}]
Take, and fix, an arbitrary mapping from the family $\mathcal{F}$ and name it $g$. Define
$\mathcal{G}:= \{g\}$ and $\mathcal{H}:= \{H_1,\dots, H_{2n+1}\}$. 
Now, the meromorphic normality of $\mathcal{F}$ is evident from the Theorem~\ref{T: finitefamily}.
\end{proof}

\begin{proof}[Proof of Corollary~\ref{C: Mon-Cara}]
Notice that for each $k\in\{1,\dots, 2n+1\}$, and each pair of mappings $f, g\in\mathcal{F}$, 
$\mathsf{supp}\,\nu(f,\,H_{k})=\mathsf{supp}\,\nu(g,\,H_{k})=\emptyset.$ Hence, 
by Corollary~\ref{C: Share 2n+1}, $\mathcal{F}$ is meromorphically normal.
\end{proof}

\section{The proof of Theorem~\ref{T: mainthm3n+1}}\label{S:mainthm3n+1}

\begin{proof}[Proof of Theorem~\ref{T: mainthm3n+1}]

Let $\{f_j\}\subset \mathcal{F}$ be an arbitrary sequence. By hypothesis, there exist 
\begin{itemize}
  \item a sequence $\{g_j\}\subset \mathcal{G}$; and

  \item $3n+1$ sequences of hyperplanes $\{H_{k,\,j}\}$, $k= 1,\dots, 3n+1$;
\end{itemize}
such that for each $j$ and each $k$
\begin{equation*}
  f_j(D)\not\subset H_{k,\,j} \text{ and } g_j(D)\not\subset H_{k,\,j}; \; \;
  D(H_{1,\,j},\dots,\,H_{3n+1,\,j})>0; \text{ and }
\end{equation*} 
\begin{equation}\label{Eq: thm2support}
  \mathsf{supp}\,\nu(f_j,\,H_{k,\,j})\subseteq \mathsf{supp}\,\nu(g_j,\,H_{k,\,j}).
\end{equation}
Let each $H_{k,\,j}$ have the representation 
\begin{equation*}
  H_{k,\,j}:= 
  \left\{[w_0: w_1: \cdots: w_n]\in\mathbb{P}^n\,\big| 
  \sum\nolimits_{l=0}^{n} a_{j,\,l}^{(k)}w_l=0\right\},\ k= 1,\dots, 3n+1.
\end{equation*}
As in the proof of Theorem~\ref{T: mainthm2n+1}, we get $3n+1$ hyperplanes with the representations
\begin{equation*}
  H_k:= 
  \left\{[w_0: w_1:\cdots: w_n]\in\mathbb{P}^n\,\big| 
  \sum\nolimits_{l=0}^{n} a_{l}^{(k)}w_l=0\right\}, \ k= 1,\dots, 3n+1.
\end{equation*}
In view of the condition $(b)$, the hyperplanes $H_1,\dots, H_{3n+1}$ 
are in general position in $\mathbb{P}^n$.
\smallskip

Again as in the proof of the Theorem~\ref{T: mainthm2n+1}, we can extract a subsequence of $\{g_j\}$
that converges compactly to a holomorphic mapping $g$. Since $H_1,\dots, H_{3n+1}$ are in general 
position in $\mathbb{P}^n$, we get a subset $L\subset\{1,\dots, 3n+1\}$ with $|L|\geq 2n+1$ 
such that $g(D)\not\subset H_k$, for all $k\in L$. Set 
\begin{equation*}
  S_k:= \mathsf{supp}\,\nu(g, H_k),\ k\in L,
\end{equation*}
and
\begin{equation*}
  E:= \bigcup_{k\in L} S_k.
\end{equation*}
By Result~\ref{R:divisor-supp}, $E$ is either an empty set or an analytic subset of 
codimension 1 in $D$. The proof follows along the same lines as the argument beginning
from the third paragraph of the proof of Theorem~\ref{T: mainthm2n+1}.
\end{proof}

\section*{Acknowledgements}

I would like to express my sincere gratitude to Gautam Bharali for helpful 
discussions during the course of this work. I would like to thank him especially for his
suggestions towards strengthening some of the results in this paper.
This work is supported by the SERB National Post-doctoral Fellowship (N-PDF)
(Grant no.~PDF/2017/001140) and a UGC CAS-II grant (Grant no.~F.510/25/CAS-II/2018(SAP-I)).

\end{document}